\documentclass[12pt,reqno]{amsart}

\usepackage{amssymb}
\usepackage{amsfonts}
\usepackage{latexsym}

\newtheorem{thm}{Theorem}[section]

\newtheorem{cor}[thm]{Corollary}

\newtheorem{lem}[thm]{Lemma}
\newtheorem{conj}[thm]{Conjecture}

\def\Z{\mathbb{Z}}
\def\al{\alpha}
\def\fl#1{\left\lfloor#1\right\rfloor}

\renewcommand{\le}{\leqslant}
\renewcommand{\ge}{\geqslant}

\numberwithin{equation}{section}

\begin{document}

%\linenumbers

\newbox\Adr
\setbox\Adr\vbox{
\centerline{\sc Victor J. W. Guo$^{1\dagger}$ and
C.~Krattenthaler$^{2\ddagger}$}
\vskip18pt
\centerline{$^1$Department of Mathematics, East China Normal
University}
\centerline{Shanghai 200062, People's Republic of China}
\centerline{Email: {\tt jwguo@math.ecnu.edu.cn}}
\centerline{WWW: \footnotesize\tt
  http://math.ecnu.edu.cn/\textasciitilde{jwguo} }
\vskip18pt
\centerline{$^2$Fakult\"at f\"ur Mathematik, Universit\"at Wien}
\centerline{Oskar-Morgenstern-Platz~1, A-1090 Vienna, Austria.}
\centerline{WWW: \footnotesize{\tt
    http://www.mat.univie.ac.at/\lower0.5ex\hbox{\~{}}kratt}.}}

\title{Some divisibility properties of binomial
and $q$-binomial coefficients}

\author[Victor J. W. Guo and C.~Krattenthaler]{\box\Adr}

\thanks{$^\dagger$Research partially supported by the Fundamental Research Funds for the Central Universities}

\thanks{$^\ddagger$Research partially supported by the Austrian
Science Foundation FWF, grants Z130-N13 and S50-N15,
the latter in the framework of the Special Research Program
``Algorithmic and Enumerative Combinatorics"}

\begin{abstract} We first prove that if $a$ has a prime
  factor not dividing $b$ then there are infinitely many
positive integers $n$ such that $\binom {an+bn} { an}$
is not divisible by $bn+1$.
This confirms a recent conjecture of Z.-W.~Sun.
Moreover, we provide some new divisibility properties of binomial
coefficients: for example, we prove that
$\binom {12n} { 3n}$ and $\binom {12n} { 4n}$ are divisible
by $6n-1$, and that
$\binom {330n} { 88n}$ is divisible by $66n-1$, for all
positive integers $n$. As we show, the latter results are in fact
consequences of divisibility and positivity results for
quotients of $q$-binomial coefficients by $q$-integers,
generalising the positivity of $q$-Catalan numbers.
We also put forward several related conjectures.
\end{abstract}



\maketitle

\section{Introduction}
The study of arithmetic properties of binomial coefficients has a long history.
In 1819, Babbage \cite{Babbage} proved the congruence
$$
\binom {2p-1} { p-1}\equiv 1 \pmod{p^2}
$$
for primes $p\geqslant 3$. In 1862, Wolstenholme
\cite{Wolstenholme} showed that the above congruence
holds modulo $p^3$ for any prime $p\geqslant 5$. See
\cite{Me} for
a historical survey on Wolstenholme's theorem. Another famous congruence is
$$
\binom {2n} { n}\equiv 0 \pmod{n+1}.
$$
The corresponding quotients,
the numbers $C_n:=\frac{1}{n+1}\binom {2n} { n}$, are called {\it
  Catalan numbers},
and they have many interesting combinatorial interpretations; see, for example,
\cite{Gould} and \cite[pp.~219--229]{Stanley}.
Recently, Ulas and Schinzel \cite{US} studied divisibility problems of
Erd\H{o}s and Straus, and of Erd\H{o}s and Graham.
In \cite{Sun2,Sun3}, Sun gave some new divisibility properties of
binomial coefficients and their products.
For example, Sun proved the following result.

\begin{thm}\label{thm:0}{\sc\cite[\sc Theorem~1.1]{Sun3}} Let $a$,
  $b$, and $n$ be positive integers. Then
\begin{align}
\binom {an+bn} { an}\equiv 0 \mod \frac{bn+1}{\gcd (a,bn+1)}. \label{eq:anbn}
\end{align}
\end{thm}

Sun also proposed the following conjecture.

\begin{conj}\label{conj}{\sc\cite[\sc Conjecture 1.1]{Sun3}}
Let $a$ and $b$ be positive integers. If\break $(bn+1)\left|
\binom {an+bn} { an}\right.$
for all sufficiently large positive integers n, then each prime factor of $a$
divides $b$. In other words, if $a$ has a prime factor not dividing
$b$, then there
are infinitely many positive integers $n$ such that
$(bn+1)\nmid\binom {an+bn} { an}$.
\end{conj}
Inspired by Conjecture~\ref{conj}, Sun \cite{Sun3} introduced a new
function $f\colon \mathbb{Z}^+\times\mathbb{Z}^+\to\mathbb{N}$.
Namely,
for positive integers $a$ and $b$, if $\binom {an+bn} { an}$ is divisible
by $bn+1$ for all $n\in\mathbb{Z}^+$, then he defined $f(a,b)=0$;
otherwise, he let
$f(a,b)$ be the smallest positive integer $n$ such that $\binom
{an+bn} { an}$ is not divisible
by $bn+1$. Using {\sl Mathematica}, Sun \cite{Sun3} computed
some values of the function $f$:
$$
f(7,36)=279,\quad f(10,192)=362,\quad f(11,100)=1187,\quad
f(22,200)=6462,\quad \ldots.
$$

The present paper serves several purposes: first of all, we give a proof
of Conjecture~\ref{conj} (see Theorem~\ref{thm:2} below);
second, we provide congruences and divisibility results similar to
the ones addressed in Theorem~\ref{thm:0} and Conjecture~\ref{conj}
(see Theorems~\ref{thm:1}--\ref{thm:3} in Section~\ref{sec:I}); third, we show
in Section~\ref{sec:q} that
among these results there is a significant number which can be
``lifted to the $q$-world;" in other words, there are several such
results which follow directly from stronger divisibility results
for $q$-polynomials. In particular, Theorem~\ref{thm:0} is an
easy consequence of Theorem~\ref{thm:anbn}, and Theorem~\ref{thm:3}
is an easy consequence of Theorem~\ref{thm:4}.
On the other hand, Theorem~\ref{conj:2} hints at the limitations
of occurrence of these divisibility phenomena.
Sections~\ref{sec:Proof1}--\ref{sec:Proof?} are devoted to
the proofs of our results in Sections~\ref{sec:I} and \ref{sec:q}.
We close our paper with Section~\ref{sec:open} by posing several open
problems.

\section{Results, I}
\label{sec:I}

Our first result is a more precise version of Conjecture~\ref{conj}.

\begin{thm}\label{thm:2}
Conjecture~{\em\ref{conj}} is true. Moreover,
if $p$ is a prime such that $p\mid a$ but $p\nmid b$,
then
$$
f(a,b)\leqslant \frac{p^{\varphi(a+b)}-1}{a+b},
$$
where $\varphi(n)$ is Euler's totient function.
\end{thm}

For the proof of the above result, we need the following theorem.

\begin{thm}\label{thm:1}
Let $a$ and $b$ be positive integers with $a>b$,
and $\beta$ an integer. Let $p$ be a
prime not dividing $a$. Then there
are infinitely many positive integers $n$ such that
$$
\binom {an} { bn+\beta}\equiv \pm 1\pmod p.
$$
\end{thm}

Our proofs of Theorems~\ref{thm:2} and \ref{thm:1} are based on Euler's
totient theorem and Lucas' classical theorem on the congruence
behaviour of binomial coefficients modulo prime numbers, see
Section~\ref{sec:Proof1}.

\medskip
In \cite[Corollary 2.3]{Guo}, the first author proved that
\begin{align}
\binom {6n} { 3n}\equiv 0 \pmod{2n-1}. \label{eq:6n3n}
\end{align}
It is easy to see that
\begin{align}
\binom {2n} { n}=2\binom {2n-1} { n}=\frac{4n-2}{n}\binom {2n-2} {
  n-1}\equiv 0 \pmod{2n-1}. \label{eq:6n3n-1}
\end{align}

The next theorem gives congruences similar to
\eqref{eq:6n3n} and \eqref{eq:6n3n-1}.

\begin{thm}\label{thm:3} Let $n$ be a positive integer. Then
\begin{align}
\binom {12n} { 3n} &\equiv \binom {12n} { 4n}  \equiv
0\pmod{6n-1}, \label{eq:6n3n-2} \\[5pt]
\binom {30n} { 5n} &\equiv 0\pmod{(10n-1)(15n-1)},  \label{eq:6n3n-3}\\[5pt]
\binom {60n} { 6n}&\equiv \binom {120n} { 40n} \equiv \binom {120n} {
  45n}  \equiv 0\pmod{30n-1}, \label{eq:6n3n-4}\\[5pt]
\binom {330n} { 88n} &\equiv 0\pmod{66n-1}. \label{eq:6n3n-5}
\end{align}
\end{thm}

We shall see that this theorem is the consequence of a stronger
result for $q$-binomial coefficients, cf.\ Theorem~\ref{thm:4}
in the next section.

It seems that there should exist many more congruences like
\eqref{eq:6n3n}--\eqref{eq:6n3n-5}. (In this direction, see
Conjecture~\ref{conj:oddp2}.)
On the other hand, we have the following negative result.

\begin{thm} \label{conj:2}
There are no positive integers $a$ and $b$ such that
$$
\binom {an+bn} { an}\equiv 0 \pmod {3n-1}
$$
for all $n\geqslant 1$.
\end{thm}

For a possible generalisation of this theorem see Conjecture~\ref{conj:oddp}
in the last section.

\section{Results, II: $q$-divisibility properties}
\label{sec:q}

Recall that the {\it $q$-binomial coefficients} (also called
{\it Gau\ss ian polynomials}) are defined by
$$
\begin{bmatrix}
n\\ k\end{bmatrix}_q
=\begin{cases}
\displaystyle\frac{(1-q^n)(1-q^{n-1})\cdots(1-q)}
{(1-q^k)(1-q^{k-1})\cdots(1-q)(1-q^{n-k})(1-q^{n-k-1})\cdots(1-q)}, &\text{if
  $0\leqslant k\leqslant n$},\\[5pt]
0,&\text{otherwise.}
\end{cases}
$$

We begin with the announced strengthening of Theorem~\ref{thm:3}.

\begin{thm}\label{thm:4} Let $n$ be a positive integer. Then all of
\begin{multline} \label{eq:6n-1}
\frac {1-q} {1-q^{6n-1}}
\begin{bmatrix} {12n}\\ { 3n}\end{bmatrix}_q ,\
\frac {1-q} {1-q^{6n-1}}
\begin{bmatrix} {12n}\\ { 4n}\end{bmatrix}_q  ,\
\frac {1-q} {1-q^{30n-1}}
\begin{bmatrix} {60n}\\ { 6n}\end{bmatrix}_q,\\
\frac {1-q} {1-q^{30n-1}}
\begin{bmatrix} {120n}\\ { 40n}\end{bmatrix}_q,\
\frac {1-q} {1-q^{30n-1}}
\begin{bmatrix} {120n}\\ { 45n}\end{bmatrix}_q  ,\
\frac {1-q} {1-q^{66n-1}}
\begin{bmatrix} {330n}\\ { 88n}\end{bmatrix}_q
\end{multline}
are polynomials in $q$ with non-negative integer coefficients.
Furthermore,
\begin{equation} \label{eq:10n15n}
\frac {(1-q)^2} {(1-q^{10n-1})(1-q^{15n-1})}
\begin{bmatrix} {30n}\\ { 5n}\end{bmatrix}_q
\end{equation}
is a polynomial in $q$.
\end{thm}

%Christian: added this sentence
For a conjectural stronger form of the last assertion in
the above theorem see Conjecture~\ref{conj:330n88n} at
the end of the paper.
\medskip

It is obvious that, when $a=b=1$, the numbers
$
\binom {an+bn} { an}/(bn+1)
$
(featured implicitly in Conjecture~\ref{conj} and in Theorem~\ref{thm:2})
reduce to the Catalan numbers $C_n$. There are various $q$-analogues
of the Catalan numbers. See F\"{u}rlinger and Hofbauer \cite{FH} for
a survey,
and see \cite{GH,Haiman,Haglund} for the so-called $q,t$-Catalan numbers.

A natural $q$-analogue of $C_n$ is
$$
C_n(q)=\frac{1-q}{1-q^{n+1}} \begin{bmatrix} 2n\\ n\end{bmatrix}_q.
$$
It is well known that the $q$-Catalan numbers $C_n(q)$ are
polynomials with non-negative integer coefficients
(see \cite{Andrews87,Andrews93,Andrews10,FH}).
Furthermore, Haiman \cite[(1.7)]{Haiman} proved
(and it follows from Lemma~\ref{lem:Andrews} below) that the polynomial
$$
\frac{1-q}{1-q^{bn+1}}\begin{bmatrix} bn+n\\ n\end{bmatrix}_q
$$
has non-negative coefficients
for all $b,n\geqslant 1$. Another generalisation of
$C_n(q)$ was introduced by the first author and Zeng \cite{GZ}:
$$
B_{n,k}(q):=\frac{1-q^k}{1-q^n}\begin{bmatrix} 2n\\ n-k\end{bmatrix}_q
=\begin{bmatrix} 2n-1\\ n-k\end{bmatrix}_q-\begin{bmatrix}
2n-1\\ n-k-1\end{bmatrix}_{q}q^k,\quad 1\leqslant k\leqslant n.
$$
They noted that the $B_{n,k}(q)$'s are polynomials in $q$, but did not
address the question whether they are polynomials with non-negative
coefficients. As the next theorem shows, this turns out to be the
case. The theorem establishes in fact a stronger non-negativity property.

\begin{thm} \label{thm:(k,n)}
Let $n$ and $k$ be non-negative integers with $0\le k\le n$.
Then
\begin{equation} \label{eq:(k,n)}
\frac {1-q^{\gcd(k,n)}} {1-q^n}\bmatrix 2n\\n-k\endbmatrix_q
\end{equation}
is a polynomial in $q$ with non-negative integer coefficients.
%Christian:
%In particular,
Consequently, also $B_{n,k}(q)=\frac{1-q^k}{1-q^n}\left[\begin{smallmatrix}
  2n\\ n-k\end{smallmatrix}\right]_q$ is a polynomial with non-negative
coefficients.
\end{thm}

%\begin{cor} \label{cor:(k,n)}
%Let $n$ and $k$ be non-negative integers with $0\le k\le n$.
%Then
%\begin{equation} \label{eq:(k,n)cor}
%\frac {1-q^k} {1-q^n}\bmatrix 2n\\n-k\endbmatrix_q
%\end{equation}
%is a polynomial in $q$ with non-negative integer coefficients.
%\end{cor}

%Recently, motivated by Bober's work \cite{Bober}, among other things,
%Warnaar and Zudilin \cite{Warnaar} proved that
%the $q$-super Catalan numbers
%$$
%\frac{(q;q)_{2m}(q;q)_{2n}}{(q;q)_{m} (q;q)_{n} (q;q)_{m+n}}
%$$
%are polynomials with non-negative integer coefficients
%for all $m,n\geqslant 0$.

Applying the inequality \cite[(2.1)]{Sun3}, we can also easily deduce that
$$
C_{a,b,n}(q):=\frac{1-q^a}{1-q^{bn+1}}\begin{bmatrix}
  an+bn\\ an\end{bmatrix}_q
$$
is a product of certain cyclotomic polynomials, and therefore a
polynomial in $q$. Again, as it turns out, all coefficients in these
polynomials are non-negative. Also here, we have
actually a stronger result,
given in the theorem below. It should be noted that it generalises
Theorem~\ref{thm:0}, the latter being obtained upon letting $q\to1$.

\begin{thm} \label{thm:anbn}
Let $a$, $b$, and $n$ be positive integers. Then
$$
\frac {1-q^{\gcd(an,bn+1)}} {1-q^{bn+1}}\bmatrix
an+bn\\an\endbmatrix_q
=\frac {1-q^{\gcd(an,bn+1)}} {1-q^{an+bn+1}}\bmatrix
an+bn+1\\an\endbmatrix_q
$$
is a polynomial in $q$ with non-negative coefficients.
\end{thm}

\begin{cor} \label{cor:anbn}
Let $a$, $b$, and $n$ be positive integers. Then
$$
\frac {1-q^a} {1-q^{bn+1}}\bmatrix
an+bn\\an\endbmatrix_q
$$
is a polynomial in $q$ with non-negative coefficients.
\end{cor}

The proofs of the results in this section are given in
Section~\ref{sec:Proof2}.

\section{Proofs of Theorems \ref{thm:2} and \ref{thm:1}}
\label{sec:Proof1}

The proof of Theorem~\ref{thm:1}
(from which subsequently Theorem~\ref{thm:2} is derived)
makes essential use of Lucas' classical theorem
on binomial coefficient congruences (see, for example,
\cite{DaWeAA,Fine,Granville,Razpet}). For the convenience of
the reader, we recall the theorem below.

\begin{thm}[\sc Lucas' theorem]Let $p$ be a prime, and let
  $a_0,b_0,\ldots, a_m,b_m\in\{0,1,\break\ldots,p-1\}$. Then
$$
\binom {a_0+a_1p+\cdots+a_m p^m} { b_0+b_1p+\cdots+b_m p^m}\equiv
\prod_{i=0}^m \binom {a_i} { b_i}\pmod p.
$$
\end{thm}

%We first give a new proof of Theorem~\ref{thm:0} as follows.
%As already noticed by Sun \cite{Sun3}, the congruence
%\eqref{eq:anbn} is equivalent to
%\begin{align}
%\frac{a}{bn+1}\binom {an+bn} { an}\in\mathbb{Z}. \label{eq:thm0}
%\end{align}
%It is easy to see that \eqref{eq:thm0} follows immediately from the relation
%\begin{align}
%\frac{a}{bn+1}\binom {an+bn} { an}=\frac{1}{n}\binom {an+bn} { an-1}
%=a\binom {an+bn} { an-1}-(a+b)\binom {an+bn-1} { an-2}. \label{eq:anbn2}
%\end{align}

\begin{proof}[Proof of Theorem \ref{thm:1}]
Note that $\gcd(p,a)=1$.
By Euler's totient theorem (see \cite{Shanks}), we have
$$
p^{\varphi(a)}-1\equiv 0\pmod a.
$$
Since $a>b>0$, there exists a positive integer $N$ such that $an>bn+\beta>0$
holds for all $n>N$.
Let $r$ be a positive integer such that $p^{r\varphi(a)}-1>aN$,
and let $n=(p^{r\varphi(a)}-1)/a$. Then, by Lucas' theorem, we have
$$
\binom {an} { bn+\beta}=\binom {p^{r\varphi(a)}-1} { bn+\beta}
\equiv \prod_{i=0}^m \binom {p-1} { b_i}\equiv (-1)^{b_0+\cdots+b_m}\pmod p,
$$
where $bn+\beta=b_0+b_1p+\cdots b_m p^m$ with $0\leqslant
b_0,\ldots,b_m\leqslant p-1$.
It is clear that there are infinitely many such $r$ and $n$. This
completes the proof.
\end{proof}

\begin{proof}[Proof of Theorem \ref{thm:2}]
Suppose that $a$ and $b$ are positive integers and
$p$ a prime such that $p\mid a$ but $p\nmid b$.
We have the decomposition
\begin{align}
\frac{1}{bn+1}\binom {an+bn} { an}
=\binom {an+bn} { an-1}-\frac{a+b}{a}\binom {an+bn-1} {
  an-2}.  \label{eq:anbn3}
\end{align}
It is clear that $p\nmid (a+b)$. By the proof of Theorem~\ref{thm:1},
if we take $n=(p^{r\varphi(a+b)}-1)/(a+b)$ ($r\geqslant 1$), then
$$
\binom {an+bn} { an-1}\equiv \pm 1\pmod p,
$$
and thus
\begin{align}
(a+b)\binom {an+bn-1} { an-2}=\frac{}{}\frac{an-1}{n}\binom {an+bn} { an-1}
\equiv\pm (a+b)\not\equiv 0\pmod p.  \label{eq:notdiv}
\end{align}
Combining \eqref{eq:anbn3} and \eqref{eq:notdiv} gives
$$
\frac{1}{bn+1}\binom {an+bn} { an}\not\in\mathbb{Z}
$$
for all $n=(p^{r\varphi(a+b)}-1)/(a+b)$ ($r=1,2,\ldots$). Namely,
Conjecture~\ref{conj} holds
and
$$
f(a,b)\leqslant \frac{p^{\varphi(a+b)}-1}{a+b},
$$
as desired.
\end{proof}

\section{Proofs of Theorems~\ref{thm:4}--\ref{thm:anbn} and
 of Corollary~\ref{cor:anbn}}
\label{sec:Proof2}

All the proofs in this section are similar in spirit.
They all draw on a lemma from
\cite[Proposition~10.1.(iii)]{ReSWAA}, which extracts the essentials
out of an argument of Andrews \cite[Proof of Theorem~2]{AndrCB}.
(To be precise, Lemma~\ref{lem:RSW} below is a slight generalisation
of \cite[Proposition~10.1.(iii)]{ReSWAA}. However, the proof from
\cite{ReSWAA} works also for this generalisation. We provide it
here for the sake of completeness.)
Recall that a  polynomial
$P(q)=
\sum _{i=0} ^{d}p_iq^i$ in $q$ of degree $d$ is called {\it reciprocal\/} if
$p_i=p_{d-i}$ for all $i$, and that it
is called {\it unimodal\/} if there is an
integer $r$ with $0\le r\le d$ and $0\le p_0\le\dots\le p_r\ge\dots\ge
p_d\ge0$.

\begin{lem} \label{lem:RSW}
Let $P(q)$ be a reciprocal and unimodal
polynomial and $m$ and $n$ positive integers with $m\le n$.
Furthermore, assume that $A(q)=\frac {1-q^m} {1-q^n}P(q)$ is a polynomial
in $q$. Then $A(q)$ has non-negative coefficients.
\end{lem}

\begin{proof}
Since $P(q)$ is unimodal, the
coefficient of $q^k$ in $(1-q^m)P(q)$ is non-negative for
$0\le k\le \deg(P)/2$. Consequently, the same must be true
for $A(q)=\frac {1-q^m} {1-q^n}P(q)$, considered as a formal power
series in $q$. However, also $A(q)$ is reciprocal, and its degree
is at most the degree of $P(q)$. Therefore the remaining coefficients
of $A(q)$ must also be non-negative.
\end{proof}

\begin{proof}[Proof of Theorem~\ref{thm:4}]
In view of Lemma~\ref{lem:RSW} and the
well-known reciprocality and unimodality of
$q$-binomial coefficients
(cf.\ \cite[Ex.~7.75.d]{Stanley}), for proving Theorem~\ref{thm:4} it
suffices to show that the expressions in \eqref{eq:6n-1} and
\eqref{eq:10n15n} are polynomials in $q$. We are going to accomplish
this by a count of the cyclotomic polynomials which divide numerators
and denominators of these expressions, respectively.

\medskip
We begin by showing that
$\frac {1-q} {1-q^{6n-1}}\left[\smallmatrix
12n\\3n\endsmallmatrix\right]_q$ is a polynomial in $q$.
We recall the well-known fact that
$$
q^n-1=\prod _{d\mid n} ^{}\Phi_d(q),
$$
where $\Phi_d(q)$ denotes the $d$-th cyclotomic polynomial in $q$.
Consequently,
$$
\frac {1-q} {1-q^{6n-1}}\begin{bmatrix} 12n\\3n\end{bmatrix}_q=
\prod _{d=2} ^{12n}\Phi_d(q)^{e_d},
$$
with
\begin{equation} \label{eq:SH}
e_d=-\chi\big(d\mid (6n-1)\big)
+\fl{\frac {12n} {d}}
-\fl{\frac {3n} {d}}
-\fl{\frac {9n} {d}},
\end{equation}
where $\chi(\mathcal S)=1$ if $\mathcal S$ is
true and $\chi(\mathcal S)=0$ otherwise.
This is clearly non-negative, unless
$d\mid (6n-1)$.

So, let us assume that $d\mid (6n-1)$, which in particular means
that $d\ge5$. Let us write
$X=\{3n/d\}$,
where $\{\al\}:=\al-\fl{\al}$ denotes the fractional
part of $\al$.
Using this notation, Equation~\eqref{eq:SH} becomes
$$e_d=-\chi\big(d\mid (2dX-1)\big)+\fl{4X}-\fl{3X}.$$
Since $0\le X<1$, we have $-1\le 2dX-1<2d$.
The only integers which are divisible by $d$ in the range
$-1,0,\dots,2d-1$ are $0$ and $d$. Hence, we must have $X=1/(2d)$
or $X=(d+1)/(2d)$. The former is impossible since $X$ is a rational
number which can be written with denominator $d$. Thus, the only
possibility left is $X=(d+1)/(2d)$. For this choice, it follows that
$\fl{4X}-\fl{3X}=2-1=1$. (Here we used that $d\ge5$.)
This proves that $e_d$ is non-negative also in this case,
and completes the proof of polynomiality of
$\frac {1-q} {1-q^{6n-1}}\left[\smallmatrix
12n\\3n\endsmallmatrix\right]_q$.

\medskip
The proof of polynomiality of
$\frac {1-q} {1-q^{6n-1}}\left[\smallmatrix
12n\\4n\endsmallmatrix\right]_q$ is completely analogous and therefore
left to the reader.

\medskip
We next turn our attention to
$\frac {1-q} {1-q^{30n-1}}\left[\smallmatrix
60n\\6n\endsmallmatrix\right]_q$. Again, we write
$$
\frac {1-q} {1-q^{30n-1}}\begin{bmatrix} 60n\\6n\end{bmatrix}_q=
\prod _{d=2} ^{60n}\Phi_d(q)^{e_d},
$$
with
\begin{equation} \label{eq:SH2}
e_d=-\chi\big(d\mid (30n-1)\big)
+\fl{\frac {60n} {d}}
-\fl{\frac {6n} {d}}
-\fl{\frac {54n} {d}}.
\end{equation}
This is clearly non-negative, unless
$d\mid (30n-1)$.

We assume $d\mid (30n-1)$ and note that this implies $d=7$ or $d\ge
11$. Here, we write $X=\{6n/d\}$.
Using this notation, Equation~\eqref{eq:SH2} becomes
$$e_d=-\chi\big(d\mid (5dX-1)\big)+\fl{10X}-\fl{9X}.$$
Since $0\le X<1$, we have $d\mid (5dX-1)$ if and only if
$X$ is one of
$$
\frac {1} {5d},\
\frac {1} {5}+\frac {1} {5d},\
\frac {2} {5}+\frac {1} {5d},\
\frac {3} {5}+\frac {1} {5d},\
\frac {4} {5}+\frac {1} {5d}.
$$
For the same reason as before, the option $X=1/(5d)$ is impossible.
For the other options, the corresponding value of $\fl{10X}-\fl{9X}$
is always $1$, except if $X=\frac {1} {5}+\frac {1} {5d}$ and $d=7$.
However, in that case, we have $X=\frac {8} {35}$, which cannot be
written with denominator $d=7$. Therefore this case can actually
not occur. This completes the proof that $e_d$ is non-negative
for all $d\ge2$, and, hence, that
$\frac {1-q} {1-q^{30n-1}}\left[\smallmatrix
60n\\6n\endsmallmatrix\right]_q$ is a polynomial in $q$.

\medskip
Proceeding in the same style,
for the proof of polynomiality of
$\frac {1-q} {1-q^{30n-1}}\left[\smallmatrix
120n\\40n\endsmallmatrix\right]_q$ we have to show that
\begin{equation} \label{eq:120n40n}
e_d=-\chi\big(d\mid (3dX-1)\big)+\fl{12X}-\fl{4X}-\fl{8X}
\end{equation}
is non-negative for all
%Christian:
$X$ of the form
$X=x/d$ with $0\le x<d$,
$x$ being integral,
%Christian:
and $d\ge2$.
Clearly, the expression in \eqref{eq:120n40n}
is non-negative, except possibly if $d\mid (3dX-1)$. With the same
reasoning as before, we see that the only cases to be examined are
$X=(d+1)/(3d)$ and $X=(2d+1)/(3d)$, where $d=7$ or $d\ge11$.
If $X=(d+1)/(3d)$, we have
$$
\fl{12X}-\fl{4X}-\fl{8X}=4-1-\fl{\frac {8} {3}+\frac {8} {3d}}.
$$
So, this will be equal to $1$, except if $d=7$. However, in that
case we have $X=\frac {8} {21}$, which cannot be written with
denominator $d=7$, a contradiction. Similarly, if $X=(2d+1)/(3d)$,
we have
$$
\fl{12X}-\fl{4X}-\fl{8X}=8-2-5=1.
$$
So, again, the exponent $e_d$ in \eqref{eq:120n40n} is non-negative,
which establishes that
$\frac {1-q} {1-q^{30n-1}}\left[\smallmatrix
120n\\40n\endsmallmatrix\right]_q$ is a polynomial in $q$.

\medskip
For the proof of polynomiality of
$\frac {1-q} {1-q^{66n-1}}\left[\smallmatrix
330n\\88n\endsmallmatrix\right]_q$ we have to show that
\begin{equation} \label{eq:330n88n}
e_d=-\chi\big(d\mid (3dX-1)\big)+\fl{15X}-\fl{4X}-\fl{11X}
\end{equation}
is non-negative for all
%Christian:
$X$ of the form
$X=x/d$ with $0\le x<d$,
$x$ being integral,
%Christian:
and $d\ge2$.
Clearly, the expression in \eqref{eq:330n88n}
is non-negative, except possibly if $d\mid (3dX-1)$. With the same
reasoning as before, we see that the only cases to be examined are
$X=(d+1)/(3d)$ and $X=(2d+1)/(3d)$, where $d=5$, $d=7$, or
%Christian:
%$d\ge11$.
$d\ge13$.
If $X=(d+1)/(3d)$, we have
$$
\fl{15X}-\fl{4X}-\fl{11X}=5+\fl{\frac {5} {d}}-1
-\fl{\frac {11} {3}+\frac {11} {3d}}.
$$
So, this will be equal to $1$, except if $d=7$.
However, again, this is an impossible case.
Similarly, if $X=(2d+1)/(3d)$, we have
$$
\fl{15X}-\fl{4X}-\fl{11X}=10+\fl{\frac {5} {d}}
-2
-\fl{\frac {22} {3}+\frac {11} {3d}}=1.
$$
So, again, the exponent $e_d$ in \eqref{eq:330n88n} is non-negative,
which establishes that
$\frac {1-q} {1-q^{66n-1}}\left[\smallmatrix
330n\\88n\endsmallmatrix\right]_q$ is a polynomial in $q$.

\medskip
Turning to \eqref{eq:10n15n}, to prove that
$\frac {(1-q)^2} {(1-q^{10n-1})(1-q^{15n-1})}\left[\smallmatrix
30n\\5n\endsmallmatrix\right]_q$ is a polynomial in $q$, we must show that
\begin{equation} \label{eq:1015}
e_d=-\chi\big(d\mid (2dX-1)\big)-\chi\big(d\mid (3dX-1)\big)
+\fl{6X}-\fl{5X}
\end{equation}
is non-negative for all
%Christian:
$X$ of the form
$X=x/d$ with $0\le x<d$,
$x$ being integral, and $d\ne5$.

First of all, we should observe that $\gcd(10n-1,15n-1)=1$, whence
the two truth functions in \eqref{eq:1015} cannot equal $1$
simultaneously.
Therefore, the expression in \eqref{eq:1015}
is non-negative, except possibly if $d\mid (2dX-1)$ or if
$d\mid (3dX-1)$. With the same
reasoning as before, we see that the only cases to be examined are
$X=(d+1)/(2d)$, $X=(d+1)/(3d)$, and $X=(2d+1)/(3d)$, where, in the
latter two cases, the choice of $d=3$ is excluded.
If $X=(d+1)/(2d)$, then
$$
\fl{6X}-\fl{5X}=3+\fl{\frac {3} {d}}
-\fl{\frac {5} {2}+\frac {5} {2d}},
$$
which always equals $1$ for $d\ge2$.
If $X=(d+1)/(3d)$, then
$$
\fl{6X}-\fl{5X}=2+\fl{\frac {2} {d}}
-\fl{\frac {5} {3}+\frac {5} {3d}},
$$
which always equals $1$ for $d=2,6,7,\dots$ (sic!).
Because of our assumptions, we do not need to consider the cases
$d=3$ and $d=5$, so the only remaining case is $d=4$. However, in
that case $X=\frac {5} {12}$, which cannot be written with denominator
$d=4$, a contradiction.
Finally, if $X=(2d+1)/(3d)$, then
$$
\fl{6X}-\fl{5X}=4+\fl{\frac {2} {d}}
-\fl{\frac {10} {3}+\frac {5} {3d}},
$$
which always equals $1$ for $d\ge2$.

So, again, the exponent $e_d$ in \eqref{eq:1015} is non-negative
in all cases,
which establishes that
$\frac {(1-q)^2} {(1-q^{10n-1})(1-q^{15n-1})}\left[\smallmatrix
30n\\15n\endsmallmatrix\right]_q$ is a polynomial in $q$.
\end{proof}

\begin{proof}[Proof of Theorem~\ref{thm:(k,n)}]
By Lemma~\ref{lem:RSW}, it suffices to establish polynomiality of
\eqref{eq:(k,n)}. When written in terms of cyclotomic polynomials,
Expression \eqref{eq:(k,n)} reads
$$
\frac {1-q^{\gcd(k,n)}} {1-q^n}\bmatrix 2n\\n-k\endbmatrix_q
=\prod _{d=2} ^{2n}\Phi_d(q)^{e_d},
$$
with
\begin{equation} \label{eq:n+k/n-k}
e_d=\chi(d\mid \gcd(k,n))-\chi(d\mid n)
+\fl{\frac {2n} {d}}
-\fl{\frac {n-k} {d}}
-\fl{\frac {n+k} {d}}.
\end{equation}
Similarly as before, let us write
$N=\{n/d\}$ and
$K=\{k/d\}$.
Using this notation, Equation~\eqref{eq:n+k/n-k} becomes
\begin{equation} \label{eq:N+K}
e_d=\chi(d\mid \gcd(k,n))-\chi(d\mid n)+\fl{2N}-\fl{N-K}-\fl{N+K}.
\end{equation}
We have to distinguish several cases. If $d\mid n$, then
$N=0$, and \eqref{eq:N+K} becomes
$$e_d=\chi(d\mid k)-1-\fl{-K}.
$$
We see that this is zero (and, hence, non-negative)
regardless whether $d\mid k$ or not.

On the other hand, if we assume that $d\nmid n$, then \eqref{eq:N+K} becomes
$$e_d=\fl{2N}-\fl{N-K}-\fl{N+K},
$$
and this is always non-negative. We have proven that \eqref{eq:(k,n)}
is indeed a polynomial in $q$.

\medskip
The statement on $B_{n,k}(q)$ follows immediately from
%Christian:
the previous result and
the fact that $\gcd(k,n)\mid k$.
\end{proof}

Finally,
Theorem~\ref{thm:anbn} will follow immediately from the following strengthening
of a non-negativity result of Andrews
\cite[Theorem~2]{AndrCB}.

\begin{lem} \label{lem:Andrews}
Let $a$ and $b$ be positive integers.
Then
\begin{equation} \label{eq:a+b}
\frac {1-q^{\gcd(a,b)}} {1-q^{a+b}}\left[\begin{matrix}
a+b\\a\end{matrix}\right]_q
\end{equation}
is a polynomial in $q$ with non-negative
integer coefficients.
\end{lem}
\begin{proof}
In view of Lemma~\ref{lem:RSW}, it suffices to show that
the expression in \eqref{eq:a+b} is a polynomial in $q$.
Again, we start with the factorisation
$$
\frac {1-q^{\gcd(a,b)}} {1-q^{a+b}}\begin{bmatrix} a+b\\a\end{bmatrix}_q=
\prod _{d=2} ^{a+b-1}\Phi_d(q)^{e_d},
$$
with
\begin{equation} \label{eq:SHab}
e_d=\chi(d\mid \gcd(a,b))
+\fl{\frac {a+b-1} {d}}
-\fl{\frac {a} {d}}
-\fl{\frac {b} {d}}.
\end{equation}
Next we write
$A=\{a/d\}$ and
$B=\{b/d\}$.
Using this notation, Equation~\eqref{eq:SHab} becomes
$$e_d=\chi(d\mid \gcd(a,b))+\fl{A+B-\frac {1} {d}}.$$
This is clearly non-negative, unless
$A=B=0$. However, in that case we have $d\mid a$ and $d\mid b$,
that is, $d\mid\gcd(a,b)$, so that $e_d$ is non-negative
also in this case.
\end{proof}

%Christian: added this
\begin{proof}[Proof of Theorem~\ref{thm:anbn}]
Replace $a$ by $an$ and $b$ by $bn+1$ in
Lemma~\ref{lem:Andrews}.
\end{proof}

\begin{proof}[Proof of Corollary~\ref{cor:anbn}]
This follows immediately from Theorem~\ref{thm:anbn} and
the fact that
$a\mid \gcd(a,bn+1)=\gcd(an,bn+1)$.
\end{proof}

\section{Proof of Theorem~\ref{conj:2}}
\label{sec:Proof?}

The following auxiliary result on the occurrence of prime numbers
congruent to~2 modulo~3 in ``small" intervals will be crucial.

%Christian: improved lower bound;
%   later text is adapted accordingly.
\begin{lem} \label{lem:p2}
If $x\ge530$, there is always at least one prime number congruent
to~$2$ modulo~$3$ contained in the interval $(x,\frac {20} {19}x)$.
\end{lem}

\begin{proof}
Let $\theta(x;3,2)$ denote the classical Chebyshev function, defined by
$$
\theta(x;3,2)=\underset{p\equiv2\ (\text{mod }3)}{\sum_{p\text{
      prime},\ p\le x}}
\log p.
$$
McCurley proved the following estimates for this
function (see \cite[Theorems~5.1 and 5.3]{McCuAA}):
\begin{align*} %\label{}
\theta(y;3,2)&<0.51\,y,\quad y>0,\\
\theta(y;3,2)&>0.49\,y,\quad y\ge3761.
\end{align*}
This implies that, for $x>3761$, we have
$$
\theta\left(\tfrac {20} {19}x;3,2\right)
-\theta(x;3,2)>0.49\cdot\tfrac {20} {19}x-0.51\,x
>0.0057\,x>1.
$$
This means that, if $x>3761$,
there must be a prime number congruent to~2 modulo~3
strictly between $x$ and $\frac {20} {19}x$.
(To be completely accurate: the above argument only shows that
such a prime number exists in the half-open interval $(x,\frac {20}
{19}x]$. However, existence in the open interval $(x,\frac {20}
  {19}x)$ can be easily established in the same manner, by slightly
  lowering the value of $\frac {20} {19}$ in the above argument.)

For the remaining
range $530\le x\le 3761$, one can verify the claim directly
using a computer.
\end{proof}

\begin{proof}[Proof of Theorem~\ref{conj:2}]
%Christian:
Given $a$ and $b$, our strategy consists in finding
a prime $p$ and a positive integer $n$ such that the
$p$-adic valuation of $\binom {an+bn} { an}/(3n-1)$
is negative, so that $3n-1$ cannot divide
. $\binom {an+bn} { an}$.
We first verified the possibility of finding such
$p$ and $n$ for $a,b\le1850$ using a computer.

\medskip
To establish the claim for the remaining values of $a$ and $b$,
we have to distinguish several cases, depending on the congruence
classes of $a$ and $b$ modulo~3 and the relative sizes of $a$ and $b$.

\medskip
First let $(a,b)\in
\{(0,0),\,(0,1),\,(1,0),\,(0,2),\,(1,1),\,(2,0)\}+(3\Z)^2$.
By Dirichlet's theorem \cite{DiriAA} (see \cite{AposAA}),
we know that there are infinitely many
primes congruent to $2$ modulo~3. Let us take such a prime $p$
with $p>a+b$, and let $3n-1=p$, that is, $n=(p+1)/3$.
Furthermore, let $v_p(\al)$ denote the $p$-adic valuation of $\al$,
that is, the maximal exponent $e$ such that $p^e$
divides $\al$. Writing $a=3a_1+a_2$ and $b=3b_1+b_2$ with $0\le
a_2,b_2\le2$, by the well-known formula of
Legendre \cite[p.~10]{LegeAA} for the $p$-adic valuation of
factorials, we then have
\begin{align} \notag
v_p\left(\frac {1} {3n-1}\binom {an+bn} { an}\right)
&=-1+\sum_{\ell\ge1}\left(
\fl{\frac {(a+b)n} {p^\ell}}
-\fl{\frac {an} {p^\ell}}
-\fl{\frac {bn} {p^\ell}}
\right)\\
&=-1+
\fl{\frac {a_2+b_2} {3}+
\frac {a+b} {3p}}
-\fl{\frac {a_2} {3}+
\frac {a} {3p}}
-\fl{\frac {b_2} {3}+
\frac {b} {3p}}.
\label{eq:vpab}
\end{align}
Since, in the current case, we have $0\le a_2+b_2\le 2$ and
$\frac {a+b} {3p}<\frac {1} {3}$, all the values of the floor
functions on the right-hand side of the above equation are zero.
Consequently, the $p$-adic valuation of $\frac {1} {3n-1}\binom
{an+bn} { an}$ equals $-1$ for our choice of $p$ and $n$, which in particular
means that $\frac {1} {3n-1}\binom {an+bn} { an}$ is not an integer.

\medskip
For the next case, consider a pair $(a,b)\in
\{(2,1),\,(2,2)\}+(3\Z)^2$. Let us first
assume that $b\le \frac {9} {10}a$.
Since we have already verified the claim for
$a,b\le 1850$, we may assume $a\ge1850$.
We now choose a prime
%Christian: (mod 3)
$p\equiv2$~(mod~$3$)
strictly between $\frac {19}
{20}a$ and $a$.
Such a prime is guaranteed to exist by Lemma~\ref{lem:p2}, because,
due to our assumption, we have $\frac {19} {20}a\ge1757.5>530$.
Furthermore, we choose $n=(p+1)/3$. We have
$$
(a+b)n\le \frac {19} {10}a \frac {p+1} {3}<\frac {2} {3}p(p+1)<p^2,
$$
and hence,
with the same notation as above, Equation~\eqref{eq:vpab} holds also
in the current case.
We have $3\le a_2+b_2\le 4$,
${a+b}\le\frac {19} {10}a<2p$, $\frac {a_2} {3}=\frac
{2} {3}$,
$\frac {1} {3}<\frac {a} {3p}\le\frac {20} {57}<\frac {2} {3}$,
hence
\begin{equation*} %\label{eq:a+bnp^2}
\fl{\frac {a_2+b_2} {3}+
\frac {a+b} {3p}}
-\fl{\frac {a_2} {3}+
\frac {a} {3p}}
-\fl{\frac {b_2} {3}+
\frac {b} {3p}}=1-1-0=0.
\end{equation*}
Consequently, again, the $p$-adic valuation of $\frac {1} {3n-1}\binom
{an+bn} { an}$ equals $-1$ for our choice of $p$ and $n$, which in particular
means that $\frac {1} {3n-1}\binom {an+bn} { an}$ is not an integer.

\medskip
Next let again $(a,b)\in
\{(2,1),\,(2,2)\}+(3\Z)^2$, but
$\frac {9} {10}a<b\le \frac {7} {5}a$.
Since we have already verified the claim for
$a,b\le 1850$, we may assume $a\ge1300$.
(If $a<1300$, then the above restriction imposes the bound
$b\le \frac {7} {5}1300=1820<1850$.)
Here, we choose a prime
%Christian: (mod 3)
$p\equiv2$~(mod~$3$)
strictly between $\frac {4}
{5}a$ and $\frac {9} {10}a$.
Such a prime is guaranteed to exist by Lemma~\ref{lem:p2}, because,
due to our assumption, we have $\frac {4} {5}a=1040>530$.
Furthermore, we choose $n=(p+1)/3$. We have still
$3\le a_2+b_2\le 4$, moreover
$2p<\frac {19} {9}p<\frac {19} {10}a\le{a+b}\le\frac {12} {5}a<
3p$, $\frac {a_2} {3}=\frac {2} {3}$,
$\frac {1} {3}<\frac {a} {3p}\le\frac {5} {12}<\frac {2} {3}$,
$\frac {1} {3}<\frac {3a} {10p}\le\frac {b} {3p}\le\frac {7a} {15p}
\le\frac {7} {12}<\frac {2} {3}$,
hence
$$
\fl{\frac {a_2+b_2} {3}+
\frac {a+b} {3p}}
-\fl{\frac {a_2} {3}+
\frac {a} {3p}}
-\fl{\frac {b_2} {3}+
\frac {b} {3p}}=1+\chi(b_2=2)-1-\chi(b_2=2)=0,
$$
implying again that
the $p$-adic valuation of $\frac {1} {3n-1}\binom
{an+bn} { an}$ equals $-1$ for our choice of $p$ and $n$, as
desired.

\medskip
The next case we discuss is $(a,b)\in
\{(2,1)\}+(3\Z)^2$ and
$\frac {7} {5}a<b\le 3a$.
Since we have already verified the claim for
$a,b\le 1850$, we may assume $b\ge1850$.
Here, we choose a prime
%Christian: (mod 3)
$p\equiv2$~(mod~$3$)
strictly between $\frac {3}
{10}b$ and $\frac {1} {3}b$.
Such a prime is guaranteed to exist by Lemma~\ref{lem:p2}, because,
due to our assumption, we have $\frac {3} {10}b=555>530$.
Furthermore, we choose $n=(p+1)/3$. Here we have
\begin{align*} \notag
v_p\left(\frac {1} {3n-1}\binom {an+bn} { an}\right)
&=-1+\sum_{\ell\ge1}\left(
\fl{\frac {(a+b)n} {p^\ell}}
-\fl{\frac {an} {p^\ell}}
-\fl{\frac {bn} {p^\ell}}
\right)\\
&=-1+\fl{\frac {(a+b)n} {p^2}}
-\fl{\frac {an} {p^2}}
-\fl{\frac {bn} {p^2}}\\
&\kern2cm
+
\fl{\frac {a_2+b_2} {3}+
\frac {a+b} {3p}}
-\fl{\frac {a_2} {3}+
\frac {a} {3p}}
-\fl{\frac {b_2} {3}+
\frac {b} {3p}}.
%\label{eq:vpab2}
\end{align*}
In this case, due to the estimations
$$p^2<b\frac {p} {3}<(a+b)n=(a+b)\frac {p+1} {3}<\frac {12} {7}b\frac {p+1} {3}<
\frac {40} {21}p(p+1)< 2p^2,\quad \text{for }p>20,$$
and
$$p^2<b\frac {p} {3}<bn<(a+b)n<2p^2,\quad \text{for }p>20,$$
we have
$$
\fl{\frac {(a+b)n} {p^2}}
-\fl{\frac {an} {p^2}}
-\fl{\frac {bn} {p^2}}=1-0-1=0.
$$
Moreover,
we have
$a_2+b_2=3$,
$3p< b<{a+b}\le\frac {12} {7}b<
\frac {120} {21}p<6p$, $\frac {a_2} {3}=\frac {2} {3}$,
$\frac {1} {3}<\frac {b} {9p}\le\frac {a} {3p}\le
\frac {5b} {21p}<\frac {50} {63}$, $\frac {b_2} {3}=\frac {1} {3}$,
$1<\frac {b} {3p}\le\frac {10} {9}$,
so that
$$
\fl{\frac {a_2+b_2} {3}+
\frac {a+b} {3p}}
-\fl{\frac {a_2} {3}+
\frac {a} {3p}}
-\fl{\frac {b_2} {3}+
\frac {b} {3p}}=2-1-1=0,
$$
implying again that
the $p$-adic valuation of $\frac {1} {3n-1}\binom
{an+bn} { an}$ equals $-1$ for our choice of $p$ and $n$, as
desired.

\medskip
The last case to be discussed is $(a,b)\in
\{(2,1)\}+(3\Z)^2$ and
$3a<b$.
Again, since we have already verified the claim for
$a,b\le 1850$, we may assume $b\ge1850$.
Here, we choose a prime
%Christian: (mod 3)
$p\equiv2$~(mod~$3$)
strictly between $\frac {5}
{11}b$ and $\frac {1} {2}b$.
Such a prime is guaranteed to exist by Lemma~\ref{lem:p2}, because,
due to our assumption, we have $\frac {5} {11}b>530$.
Furthermore, we choose $n=(p+1)/3$. Since
$$
(a+b)n<\frac {4} {3}b\frac {p+1} {3}<\frac {44} {45}p(p+1)<p^2,\quad
\text{for }p>44,
$$
for the $p$-adic valuation of $\frac {1} {3n-1}\binom
{an+bn} { an}$ there holds again \eqref{eq:vpab}.
In the current case, we have furthermore
$a_2+b_2=3$,
${a+b}\le\frac {4} {3}b<
\frac {44} {15}p<3p$, $\frac {a_2} {3}=\frac {2} {3}$,
$\frac {a} {3p}<\frac {b} {9p}<\frac {11} {45}<\frac {1} {3}$,
$\frac {b_2} {3}=\frac {1} {3}$,
$\frac {2} {3}<\frac {b} {3p}\le\frac {11} {15}$,
and hence
$$
\fl{\frac {a_2+b_2} {3}+
\frac {a+b} {3p}}
-\fl{\frac {a_2} {3}+
\frac {a} {3p}}
-\fl{\frac {b_2} {3}+
\frac {b} {3p}}=1-0-1=0,
$$
implying also here that
the $p$-adic valuation of $\frac {1} {3n-1}\binom
{an+bn} { an}$ equals $-1$ for our choice of $p$ and $n$, as
desired.

\medskip
%Christian:
We have now covered all possible cases (in particular,
by symmetry in $a$ and $b$, we also covered the case
$(a,b)\in \{(1,2)\}+(3\Z)^2$), and hence
this concludes the proof of the theorem.
\end{proof}

\section{Concluding remarks and open problems}
\label{sec:open}

In the proof of Theorem~\ref{thm:2}, assume that $s$ is the smallest
positive integer
such that $(a+b)\mid (p^s-1)$. Then we obtain the stronger
inequality
\begin{align}
f(a,b)\leqslant \frac{p^{s}-1}{a+b}. \label{eq:ineq}
\end{align}
It is easily seen that $s\mid \varphi(a+b)$. However, such an upper
bound is still likely much
larger than the exact value of $f(a,b)$ given by Sun \cite{Sun3}.
For example, the inequality \eqref{eq:ineq} gives
\begin{align*}
f(7,36) &\leqslant  \frac{7^{6}-1}{43}=2736, \\[5pt]
f(10,192)&\leqslant \frac{5^{25}-1}{202}=1475362494440362, \\[5pt]
f(11,100)&\leqslant \frac{11^{6}-1}{111}=15960,\\[5pt]
f(22,200)&\leqslant \frac{11^{6}-1}{222}=7980, \\
f(1999,2011)& \leqslant \frac{1999^{400}-1}{4010}\thickapprox
5.272\times 10^{1316}.
\end{align*}

It seems that Theorem~\ref{thm:1} can be further generalised in the
following way.

\begin{conj}
Let $a$ and $b$ be positive integers with $a>b$,
and let $\alpha$ and $\beta$ be integers. Furthermore, let $p$ be
a prime such that $\gcd(p,a)=1$.
Then for each $r=0,1,\ldots,p-1$, there
are infinitely many positive integers $n$ such that
$$
\binom {an+\alpha} { bn+\beta}\equiv r\pmod p.
$$
\end{conj}

In relation to Theorem~\ref{thm:3},
we propose the following two conjectures, the first one generalising
Theorem~\ref{conj:2}.

\begin{conj}\label{conj:oddp}
For any odd prime $p$, there are no positive integers $a>b$ such that
$$
\binom {an} { bn}\equiv 0 \pmod {pn-1}
$$
for all $n\geqslant 1$.
\end{conj}

\begin{conj}\label{conj:oddp2}
For any positive integer $m$, there are positive integers $a$ and $b$
such that $am>b$ and
$$
\binom {amn} { bn}\equiv 0 \pmod {an-1}
$$
for all $n\geqslant 1$.
\end{conj}

Note that the congruences \eqref{eq:6n3n}--\eqref{eq:6n3n-5}
imply that Conjecture~\ref{conj:oddp2} is true for $1\leqslant m\leqslant 5$.

It seems clear that, for each {\it specific prime $p$},
a proof of Conjecture~\ref{conj:oddp} in the style of the
proof of Theorem~\ref{conj:2} in Section~\ref{sec:Proof?}
can be given. On the other hand, a proof for {\it arbitrary $p$}
will likely require a new idea.

\medskip
We end the paper with the following conjecture, strengthening
the last part of Theorem~\ref{thm:4}.

\begin{conj}\label{conj:330n88n}
For all positive integers $n$ and non-negative integers $k$ with
$0\le k\le
%5n 25n +2-(10n-1)-(15n-1)$,
125n^2 -25n+4$,
the coefficient of $q^k$ of the polynomial
$$
\frac {(1-q)^2} {(1-q^{10n-1})(1-q^{15n-1})}
\begin{bmatrix} {30n}\\ { 5n}\end{bmatrix}_q
$$
is non-negative,
except for $k=1$ and $k=125n^2 -25n+3$, in which case
the corresponding coefficient equals $-1$.
\end{conj}

\section*{Note}
Conjecture~\ref{conj:oddp} was proved by Madjid Mirzavaziri and
the second author.
Conjecture~\ref{conj:oddp2} was proved by Madjid Mirzavaziri.

\section*{Acknowledgement}
We thank Pierre Dusart for very helpful discussion concerning
Lemma~\ref{lem:p2}.

\end{document}